\documentclass[11pt]{article}

\usepackage[all,cmtip]{xy}
\usepackage{amssymb,mathrsfs}
\usepackage{amsthm} 
\usepackage{mathptmx}
\usepackage{latexsym}
\usepackage[utf8]{inputenc}
\usepackage[english]{babel}

\def\Mod{{\rm \mbox{-}Mod}}

\def\Hom{{\rm Hom}}
\def\Im{{\rm Im\ }}
\def\Ker{{\rm Ker}}

\def\id{{\rm id}}

\def\A{{\mathscr{A}}}
\def\C{{\mathscr{C}}}
\def\K{{\mathscr{K}}}

\newtheorem{thm}{\bf Theorem}[section]
\newtheorem{cor}[thm]{\bf Corollary}
\newtheorem{lem}[thm]{\bf Lemma}
\newtheorem{prop}[thm]{\bf Proposition}

\newtheorem{rem}[thm]{\bf Remark}

\newtheorem{ex}[thm]{\bf Example}

\begin{document}
\title{A new approach to projectivity in the categories of complexes}
\author{Driss Bennis, J. R. Garc\'{\i}a Rozas,\\ Hanane Ouberka and Luis Oyonarte}
\date{}
\maketitle

\noindent{\large\bf Abstract.} Recently, several authors have adopted new alternative approaches in the study of some classical notions of modules. Among them, we find the notion of subprojectivity which was introduced to measure in a way the degree of projectivity of modules. The study of subprojectivity has recently been extended to the context of abelian categories, which has brought to light some interesting new aspects. For instance, in the category of complexes, it gives a new way to measure, among other things, the exactness of complexes. In this paper, we  prove that the subprojectivity notion provides a new sight of null-homotopic morphisms in the category of complexes. This will be proven through two main results. Moreover, various results which emphasize the importance  of subprojectivity in the category of complexes are also given. Namely, we give some applications by characterizing some classical rings and establish various examples that allow us to reflect the scope and limits of our results.

\bigskip



\section{Introduction}
\hskip .5cm Throughout the paper $R$ will denote an associative (non necessarily commutative) ring with unit. The category of left $R$-modules will be denoted by $R\Mod$. Modules are, unless otherwise explicitly stated, left $R$-modules.

The notion of subprojectivity was introduced in \cite{Sergio} as a new treatment in the analysis of the projectivity of a module. However, the study of the subprojectivity goes beyond that goal and, indeed, provides, among other things, a new and interesting perspective on some other known notions. In this way, an alternative perspective on the projectivity of an object of an abelian category $\A$ with enough projectives was investigated in \cite{SubprojAb} where, in addition, it was shown that subprojectivity can be used to measure characteristics different from the projectivity and that subprojectivity domains may not be restricted to a single object. On the contrary, the subprojectivity domains of a whole class of objects can be computed, giving rise to very interesting characterizations. For instance, the subprojectivity domain of the whole class of DG-projective complexes is very useful to measure the exactness of complexes (see \cite[Proposition 2.5]{SubprojAb}).

Recall that, given two objects $M$ and $N$ of $\A$, $M$ is said to be $N$-subprojective if for every epimorphism $g:B\to N$ and every morphism $f:M\to  N$, there exists a morphism $h:M\to B$ such that $gh=f$, or equivalently, if every morphism $M\to N$ factors through a projective object (see \cite[Proposition 2.7]{SubprojAb}). The subprojectivity domain of any object $M$, denoted ${\underline {\mathfrak{Pr}}}_{\A}^{-1}(M)$, is defined as the class of all objects $N$ such that $M$ is $N$-subprojective, and the subprojectivity domain of a whole class $\mathfrak{C}$ of $\A$, ${\underline {\mathfrak{Pr}}}_{\A}^{-1}(\mathfrak{C})$, is defined as the class of objects $N$ such that every $C$ of $\mathfrak{C}$ is $N$-subprojective.
 
In this paper we go deeper in the investigation of subprojectivity in the category of complexes. In this sense, when studying subprojectivity of complexes, it is observed that the concept of subprojectivity is relatively closely linked to that of null-homotopy of morphisms. Therefore, what we intend in the two main results of this paper (Theorems \ref{thm-4-2} and \ref{thm-4-1}) is to deepen the understanding of this relationship. Namely, in Theorem \ref{thm-4-2}, we prove that if $N_{n}\in{\underline {\mathfrak{Pr}}}_{R\Mod}^{-1}(M_n)$ for every $n \in \mathbb{Z}$, then $N\in{\underline {\mathfrak{Pr}}}_{\C(R)}^{-1}(M)$ ($\C(R)$ is the category of complexes) if and only if $\Hom_{\K(R)}(M[-1],K)=0$ ($\K(R)$ is the homotopy category) for every short exact sequence of complexes $0\to K \to P \to N \to0$ with $P$ projective. The proof of this theorem is based on a new characterization of the subprojectivity of an object in any abelian category with enough projectives in terms of the splitting of some particular short exact sequences (Proposition \ref{prop-pull}).  

The second main result of the paper (Theorem \ref{thm-4-1}) assures that for any two complexes $M$ and $N$ with $N_{n+1} \in {\underline {\mathfrak{Pr}}}_{R\Mod}^{-1}(M_n)$ for every $n\in \mathbb{Z}$, the conditions $N \in {\underline {\mathfrak{Pr}}}_{\C(R)}^{-1}(M)$ and $\Hom_{\K(R)}(M,N)=0$ are equivalent. This time, the idea is based on a new characterization of subprojectivity in terms of factorizations by contractible complexes (Proposition \ref{pro-cont2}). 

Theorem \ref{thm-4-1} allows us to determine exactly when a complex $N$ is in the subprojectivity domain of all the shifts $M[n]$ of a given complex $M$ (Proposition \ref{prop-shift}), which, at the same time helps in characterizing subprojectivity domains of complexes of the form $\oplus_{n\in \mathbb{Z}}\overline{M}[n]$ (Proposition \ref{lem-discs}) and of the form $\oplus_{n\in \mathbb{Z}}\underline{M}[n]$ (Proposition \ref{prop-sph}) for a given module $M$. A particular case of Proposition \ref{prop-sph} typifies exact complexes in terms of subprojectivity in the following sense: $N$ is exact if and only if $N\in{\underline {\mathfrak{Pr}}}_{\C(R)}^{-1}(\underline{R}[n])$ for every $n\in \mathbb{Z}$ (Corollary \ref{cor-exac}). Motivated by this result, we asked whether subprojectivity can measure the exactness of a complex $N$ at each $N_i$. In fact, we prove that, for any  complex  $N$ and any $n\in \mathbb{Z}$, 
 $N \in {\underline {\mathfrak{Pr}}}_{\C(R)}^{-1}(\underline{R}[n])$  if and only if $H_n(N)=0$ (see Proposition \ref{prop-spherR}). This result allows us to answer two interesting questions. Namely, we provide  an example showing that the subprojectivity domains are not closed under kernel of epimorphisms (see Example \ref{exmp-1-spherR}). And, we give an example showing that the equivalence of Theorem \ref{thm-4-2} mentioned above does not hold in general if we replace  the condition  ``$P$ is projective''  with $P\in {\underline {\mathfrak{Pr}}}_{\C(R)}^{-1}(M)$  (see Remark \ref{rem-cond1} and Example \ref{exmp-2-spherR}).

It is worth mentioning that the necessity and the importance of the conditions given in the main Theorems \ref{thm-4-2} and \ref{thm-4-1} are deeply discussed in Propositions \ref{prop-compon} and \ref{prop-semisimple}, respectively, and Example \ref{ex-2main1}.   

All this is done in Section \ref{sect:main}. \\

Finally,   Section \ref{sect:aplic} is devoted to some applications. Namely, we give, as consequences of Theorem \ref{thm-4-2}, some new characterizations of some classical rings. In Proposition \ref{prop-hered} we characterize left hereditary rings in terms of subprojectivity as those rings for which every subcomplex of a DG-projective complex is DG-projective. Furthermore, we do it without the condition ``Every exact complex of projective modules is projective'' needed in \cite[Proposition 2.3]{Yang2}.

Following the same context, subprojectivity also makes it possible to characterize rings of weak global dimension at most $1$, and using subprojectivity domains we prove that these rings are the ones over which subcomplexes of DG-flat complexes are always also DG-flat (Proposition \ref{prop-Weak-dim}).  As a consequence, left semi-hereditary rings are also characterized in terms of subprojectivity (Corollary  \ref{cor-semi-her}).

Finally, it is worth noting that semisimple rings are also characterized in terms of subprojectivity. In fact, this was a consequence of the study of the condition  ``$N_{n+1} \in {\underline {\mathfrak{Pr}}}_{R\Mod}^{-1}(M_n)$ for every $n \in \mathbb{Z}$"  assumed in Theorem \ref{thm-4-1}. Namely, we prove  that the ring $R$ must be semisimple when this condition implies the condition   $N \in {\underline {\mathfrak{Pr}}}_{\C(R)}^{-1}(M)$ for every two complexes $M$ and $N$ (Proposition \ref{prop-semisimple}).

\section{Preliminaries}	

\hskip .5cm In this section we fix some notations and recall some definitions and basic results     that will be used throughout this article.\\

Recall that, for two objects $M$ and $N$ of an abelian category with enough projectives, $M$ is said to be $N$-sub\-pro\-jec\-ti\-ve if for every epimorphism $g:B \to N $ and every morphism $f:M\to N$,  there exists a morphism $h:M\to B$ such that $gh=f$. Equivalently, $M$ is $N$-subprojective if and only if every morphism $M\to N$ factors through a projective object (\cite[Proposition 2.7]{SubprojAb}). The subprojectivity domain of any object $M$ is defined as $${\underline {\mathfrak{Pr}}}_{\A}^{-1}(M)=\left\{ N\in\A;\ M\ \mbox{is}\ N\mbox{-subprojective}\right\}.$$ 

By a complex $X$ of modules we mean a sequence of modules and morphisms $$\xymatrix{\cdots \ar[r] & X_{2}\ar[r]^{d_2} & X_{1}\ar[r]^{d_1}  &X_{0}\ar[r]^{d_0} &X_{-1}\ar[r]^{d_{-1}} &X_{-2}\ar[r]&\cdots}$$ such that $d_nd_{n+1}=0$ for all $n\in \mathbb{Z}$. If $ \Im d_{n+1}=\ker\ d_n$ for all $n \in\mathbb{Z}$ then we say that $X$ is exact, and given an $R$-module $M$, $X$ is said to be $\Hom_R(M,-)$-exact if the complex of abelian groups $\Hom_R(M,N)$ is exact. We denote by $\epsilon_n^X: X_{n} \to \Im d_{n}$ the canonical epimorphism and by $\mu_n^X : \Ker(d_{n-1}) \to X_{n-1}$ the canonical monomorphism.

The $n^{th}$ boundary (respectively, cycle, homology) of a complex $X$ is defined as $\Im\, d_{n+1}^X$ (respectively, $\Ker\, d_n^X$,  $ \Ker\, d_n / \Im\, d_{n+1}$) and it is denoted by $B_n(X)$ (respectively, $Z_n(X)$, $H_n(X)$).

Throughout the paper, we use the following particular kind of complexes:

\begin{description}
\item[Disc complex.] Given a module $M$, we denote by $\overline{M}$ the complex $$\xymatrix{\cdots \ar[r] &0\ar[r] &M\ar[r]^{\id_M}  &M\ar[r] &0\ar[r] &\cdots}$$ with all terms $0$ except $M$ in the degrees $1$ and $0$.
\item[Sphere complex.] Also, for a module $M$, we denote by $\underline{M}$ the complex $$\xymatrix@1{\cdots \ar[r] &0\ar[r]  &M\ar[r] &0\ar[r] &\cdots}$$ with all terms $0$ except $M$ in the degree $0$.
\item[Shift complex.] Let $X$ be a complex with differential $d^X$ and fix an integer $n$. We denote by $X[n]$ the complex consisting of $X_{i-n}$ in degree $i$ with differential $(-1)^nd_{i-n}^X$.
\end{description}

Now, by a morphism of complexes $f:X\rightarrow Y$ we mean a family of morphisms  $f_n: X_n\rightarrow Y_n$ such that $d^{Y}_nf_n=f_{n-1}d^{X}_n$ for all $n\in \mathbb{Z}$. The category of complexes of $R$-modules will be denoted by $\C(R)$. For two complexes $X$ and $Y$, we use $\Hom_{\C(R)}(X,Y)$ to present the group of all morphisms of complexes from $X$ to $Y$.
 
A morphism of complexes $f:X \rightarrow Y$ is said to be null-homotopic if, for all $n\in\mathbb{Z}$, there exist morphisms $s_n:X_n \rightarrow Y_{n+1}$ such that for any $n$ we have $f_n=d^{Y}_{n+1}s_n+s_{n-1}d_n^{X}$, and then we say that $f$ is null-homotopic by $s$. In particular, for a complex $X$, $\id_X$ is null-homotopic if and only if $X$ is of the form $\oplus_{n \in \mathbb{Z}} \overline{M_n}[n]$ for some family of modules $M_n\in R$-Mod. A complex of this special type is called contractible.
 
Two morphisms of complexes $f$ and $g$ are homotopic, $f \sim g$ in symbols, if $f-g$ is null-homotopic. The relation $f \sim g$ is an equivalence relation. The homotopy category $\K(R)$ is defined as the one having the same objects as $\C(R)$, and which morphisms are  homotopy equivalence classes of morphisms in $\C(R)$.\\
 
For complexes $X$ and $Y$, we let $\Hom^{\bullet}(X,Y)$ denote the complex of abelian groups with $$ \Hom^{\bullet}(X,Y)_{n}=\prod_{i\in \mathbb{ Z}}\Hom_R(X_{i},Y_{i+n})$$ and $$ d_{n}^{\Hom^{\bullet}(X,Y)}(\psi)=(d^Y_{i+n}\psi_i -(-1)^{n}\psi_{i-1}d_i^X)_{i\in \mathbb{ Z}}.$$ 

Note that for every $n \in \mathbb{Z}$, $$Z_n(\Hom^{\bullet}(X,Y))=\Hom_{\C(R)}(X[n],Y)=\Hom_{\C(R)}(X,Y[-n])$$ and $$H_n(\Hom^{\bullet}(X,Y))=\Hom_{\K(R)}(X[n],Y)=\Hom_{\K(R)}(X,Y[-n]).$$ For every complex $X$, $\Hom^{\bullet}(X,-)$ is a left exact functor from the category of complexes of modules to the category of complexes of abelian groups.

\section{Subprojectivity and null-homotopy} \label{sect:main}
\hskip .5cm As mentioned in the introduction, subprojectivity of complexes is closely related with null-homotopy of morphisms of complexes and kernels of epimorphisms. The aim of this section is to deepen the understanding of this relationship. 

We start with a new characterization of subprojectivity in terms of splitting short exact sequences which will be considered somehow as the subprojectivity analogue of the classical characterization of projectivity.  

We fix the following notation: the pullback of two morphisms $g:C \to B$ and $f:A \to B$ will be denoted by $(D,g',f')$.

\begin{prop}\label{prop-pull}
Let $\A$ be an abelian category with enough projectives. If $M$ and $N$ are two objects of $\A$, the following conditions are equivalent.
\begin{enumerate} 
\item $N\in{\underline{\mathfrak{Pr}}}_{\mathscr{A}}^{-1}(M)$.
\item For every epimorphism $g:K \to N$ and every morphism $f:M \to N$. The epimorphism $g':D\to M$ given by the pullback $(D,g',f')$ of $g$ and $f$, splits.
\item There exists an epimorphism $g:P\to N$ with $P$ projective such that for every morphism $f:M \to N$, the epimorphism $g':D\to M$ given by the pullback $(D,g',f')$ of $g$ and $f$, splits.
\item There exists an epimorphism $g:P\to N$ with $P\in{\underline{\mathfrak{Pr}}}_{\mathscr{A}}^{-1}(M)$ such that for every morphism $f:M \to N$, the epimorphism $g':D\to M$ given by the pullback $(D,g',f')$ of $g$ and $f$, splits.
\end{enumerate}
\end{prop}
\begin{proof} $1.\Rightarrow 2.$ Let $g:K \to N$ be an epimorphism, $f:M \to N$ be a morphism and $(D,g',f')$ be their pullback. Since $N\in{\underline {\mathfrak{Pr}}}_{\mathscr{A}}^{-1}(M)$, there exists a morphism $h:M \to K$ such that the following diagram commutes $$\xymatrix{M \ar[ddr]_h\ar[drr]^{\id_M}& & \\ & D\ar[r]_{g'}\ar[d]^{f'}&M\ar[d]^f \\ & K\ar[r]_g & N}$$ Then, by the universal property of pullbacks, there exists a morphism $k:M\to D$ such that $g'k=id_M$. Hence $g'$ splits, as desired.

$2.\Rightarrow 3.$ This is clear since the category $\A$ is supposed to have enough projectives.

$3. \Rightarrow 4.$ This is clear since every projective complex belongs to ${\underline {\mathfrak{Pr}}}_{\mathscr{A}}^{-1}(M)$.

$4. \Rightarrow 1.$ Let $g:P \to N$ be the epimorphism of statement 4., $f:M \to N$ be a morphism and $(D,g',f')$ their pullback $$\xymatrix{D\ar[r]^{g'}\ar[d]_{f'} & M\ar[d]^f \\ P\ar[r]^g & N}$$ Then, by assumption, there exists a morphism $h:M\to D$ such that  $g'h=id_M$, hence $f= fg'h=gf'h$. Therefore, $N \in {\underline {\mathfrak{Pr}}}_{\mathscr{A}}^{-1}(M)$ (see \cite[Proposition 2.2]{SubprojAb}).
\end{proof}

The following lemma will be useful in the proof of Theorem \ref{thm-4-2}.

\begin{lem}\label{lem-pull}
If $(D,g',f')$ is the pullback of two morphisms of complexes $g:C\to B$ and $f:A\to B$, then $(D_n,g'_n,f'_n)$ is the pullback of $g_n:C_n \to B_n$ and $f_n:A_n \to B_n$ for every $n\in \mathbb{Z}$.
\end{lem}
\begin{proof}
Let $\alpha :X \to A_n$ and $\beta:X \to C_n$ be two morphisms of modules such that $f_n \alpha= g_n \beta$ and consider the two morphisms of complexes $\overline{\alpha}:\overline{X}[n-1]\to A$ and $\overline{\beta}:\overline{X}[n-1]\to C$ induced by $\alpha$ and $\beta$, respectively. It is straightforward to verify that $f\overline{\alpha}=g\overline{\beta}$, so there exists a unique morphism of complexes $h:\overline{X}[n-1]\to D$ such that $g'h=\overline{\alpha}$ and $f'h=\overline{\beta}$. Then, $g'_nh_n=\alpha$ and $f'_nh_n=\beta$. 

The unicity of $h_n:X \to D_n$ comes from the unicity of $h$.
\end{proof}

Now, we give the first main result of the paper.

\begin{thm}\label{thm-4-2}
Let $M$ and $N$ be two complexes such that $N_n \in {\underline {\mathfrak{Pr}}}_{R\Mod}^{-1}(M_n)$ for every $n\in \mathbb{Z}$. Then, the following statements are equivalent.
\begin{enumerate}
\item $N \in {\underline {\mathfrak{Pr}}}_{\C(R)}^{-1}(M)$.
\item For every short exact sequence $0\to K \to P \to N \to0$ with $P$ projective, the equation $\Hom_{\K(R)}(M[-1],K)=0$ holds.
\item There exists a short exact sequence $0\to K \to P \to N \to 0$ with $P$ projective such that $\Hom_{\K(R)}(M[-1],K)=0$.
\item There exists a short exact sequence $0\to K \to P \to N \to 0$ with $P\in {\underline {\mathfrak{Pr}}}_{\C(R)}^{-1}(M)$ such that $\Hom_{\K(R)}(M[-1],K)=0$.
\end{enumerate}
\end{thm}
\begin{proof}
$1. \Rightarrow 2.$ Let  $0\to K \to P \to N \to0$ be a short exact sequence with $P$ projective and consider the following commutative diagram with exact rows $$\xymatrix{ & 0\ar[d] & 0\ar[d] & 0\ar[d] & \\ 0 \ar[r] & Z_{0}(\Hom^{\bullet}(M,K)) \ar[r] \ar[d] & \Hom^{\bullet}(M,K)_0 \ar[r] \ar[d] & B_{-1}(\Hom^{\bullet}(M,K)) \ar[r] \ar[d] & 0 \\ 0 \ar[r] & Z_{0}(\Hom^{\bullet}(M,P)) \ar[r] \ar[d] & \Hom^{\bullet}(M,P)_0 \ar[r] \ar[d] & B_{-1}(\Hom^{\bullet}(M,P)) \ar[r] \ar[d] & 0 \\ 0 \ar[r] & Z_{0}(\Hom^{\bullet}(M,N)) \ar[r] \ar[d] & \Hom^{\bullet}(M,N)_0 \ar[r] \ar[d] & B_{-1}(\Hom^{\bullet}(M,N)) \ar[r] \ar[d] & 0 \\ & 0 & 0 & 0}$$

The first and second columns are exact since $N \in {\underline {\mathfrak{Pr}}}_{\C(R)}^{-1}(M)$ and $N_n \in {\underline {\mathfrak{Pr}}}_{R\Mod}^{-1}(M_n)$ for every $n\in \mathbb{Z}$, respectively. Hence, the third column is also exact. 

Now, applying the Snake Lemma to the following commutative diagram with exact rows and columns $$\xymatrix{ & 0\ar[d] & 0\ar[d] & 0\ar[d] & \\ 0\ar[r] & B_{-1}(\Hom^{\bullet}(M,K))\ar[r]\ar[d] & B_{-1}(\Hom^{\bullet}(M,P))\ar[r] \ar[d] & B_{-1}(\Hom^{\bullet} (M,N)) \ar[r]\ar[d] & 0 \\ 0\ar[r] & Z_{-1}(\Hom^{\bullet} (M,K))\ar[r]& Z_{-1}(\Hom^{\bullet} (M,P))\ar[r]& Z_{-1}(\Hom^{\bullet} (M,N)) &}$$ we get the exact sequence $$\xymatrix{ 0\ar[r]& H_{-1}(\Hom^{\bullet} (M,K))\ar[r]& H_{-1}(\Hom^{\bullet} (M,P))\ar[r]& H_{-1}(\Hom^{\bullet} (M,N))&}.$$

By \cite[Corollary 3.5]{Gillespie2} and since $$H_{-1}(\Hom^{\bullet} (M,P))=\Hom_{\K(R)}(M[-1],P)=0,$$  we get that   $\Hom_{\K(R)}(M[-1],K)=H_{-1}(\Hom^{\bullet} (M,K))=0$.

$2.\Rightarrow 3.$ Clear since the category of complexes has enough projectives.

$3.\Rightarrow 4.$ This is clear since every projective complex belongs to ${\underline {\mathfrak{Pr}}}_{\C(R)}^{-1}(M)$.

$4.\Rightarrow 1.$ Let $0\to  K\to P\to N \to 0$ be the short exact sequence of statement 4., $f:M\to N$ be any morphism of complexes and consider the following pullback diagram $$\xymatrix{ 0\ar[r]& K\ar[r]\ar@{=}[d]& D\ar[r] \ar[d]& M \ar[r]\ar[d]^f & 0 \\	0\ar[r]& K\ar[r]& P\ar[r]& N \ar[r] & 0}$$

For every $n\in \mathbb{Z}$, $D_n$ is a pullback by Lemma \ref{lem-pull}, so by assumption and Proposition \ref{prop-pull} the short exact sequence $ 0\to K\to D\to M \to 0$ splits at the module level. Then, this sequence is equivalent to a short exact sequence $0\to K\to M(g)\to M \to 0 $ being $M(g)$ the mapping cone of a morphism $g:M[-1] \to K$ (see \cite[Section 3.3]{Enochs}). But $g: M[-1] \to K$ is null-homotopic by assumption so $0\to K\to M(g)\to M \to 0$ splits (see \cite[Proposition 3.3.2]{Enochs}). Therefore, the sequence $ 0\to K\to D\to M \to 0$ splits too and then $N\in{\underline {\mathfrak{Pr}}}_{\C(R)}^{-1}(M)$ by Proposition \ref{prop-pull}.
\end{proof}

\begin{rem}\label{rem-cond1} It is natural to ask whether,  
as in  the case of exact sequences $0\to K \to P \to N \to0$ with  $P$ projective,   the statements of Theorem \ref{thm-4-2} are   equivalent to the following: ``For every short exact sequence $0\to K \to P \to N \to0$ with  $ P \in {\underline {\mathfrak{Pr}}}_{\C(R)}^{-1}( M)$, the equation $\Hom_{\K(R)}(M[-1],K)=0$ holds''.   We will see in  Example    \ref{exmp-2-spherR}  that  they are not equivalent.
\end{rem}

Given two complexes $M$ and $N$, it is natural to ask if $ N \in {\underline {\mathfrak{Pr}}}_{\C(R)}^{-1}(M)$ is sufficient to get that, for every $n \in \mathbb{Z}$, $N_{n} \in {\underline {\mathfrak{Pr}}}_{R\Mod}^{-1}(M_n)$. This is not true in general. Indeed, we can always consider two modules $X$ and $Y$ with $Y\notin {\underline {\mathfrak{Pr}}}_{R\Mod}^{-1}(Y)$, while it is clear that we always have $\underline{Y} \in {\underline {\mathfrak{Pr}}}_{\C(R)}^{-1}(\overline{X})$ since every morphism $\overline{X} \to \underline{Y}$ is zero. Nevertheless, the answer to the question would be positive if we assume, furthermore, that $N$ belongs to ${\underline {\mathfrak{Pr}}}_{\C(R)}^{-1}(M[-1])$.

\begin{prop}\label{prop-compon}
Let $M$ and $N$ be two complexes. If $ N \in {\underline {\mathfrak{Pr}}}_{\C(R)}^{-1}( M[-1]) \bigcap {\underline {\mathfrak{Pr}}}_{\C(R)}^{-1}(M)$, then  $N_{n} \in {\underline {\mathfrak{Pr}}}_{R\Mod}^{-1}(M_n)$ for every $n \in \mathbb{Z}$.
\end{prop}
\begin{proof}
Let $P$ be a projective complex and $P\to N$ be an epimorphism of complexes. Since  $ N, P \in {\underline {\mathfrak{Pr}}}_{\C(R)}^{-1}(M[-1])$, $\Hom_{\K(R)}(M[-1],P)=\Hom_{\K(R)}(M[-1],N)=0$ by \cite[Corollary 3.5]{Gillespie2}. So, the horizontal maps of the following commutative diagram are isomorphisms $$\xymatrix{  B_{-1}(\Hom^{\bullet} (M,P)) \ar[r] \ar[d] & Z_{-1}(\Hom^{\bullet} (M,P)) \ar[d] \\ B_{-1}(\Hom^{\bullet} (M,N)) \ar[r] & Z_{-1}(\Hom^{\bullet} (M,N))}$$

The morphism $$Z_{-1}(\Hom^{\bullet} (M,P)) \to Z_{-1}(\Hom^{\bullet} (M,N))$$ is epic since $ N \in {\underline {\mathfrak{Pr}}}_{\C(R)}^{-1}(M[-1])$, so the map $B_{-1}(\Hom^{\bullet} (M,P)) \to B_{-1}(\Hom^{\bullet} (M,N))$ must also be epic. 

Now, consider the following commutative diagram with exact rows: $$\xymatrix{0 \ar[r] & Z_{0}(\Hom^{\bullet} (M,P)) \ar[r] \ar[d]  & \Hom^{\bullet} (M,P)_0 \ar[r] \ar[d] & B_{-1}(\Hom^{\bullet} (M,P)) \ar[r] \ar[d]& 0 \\ 0 \ar[r]& Z_0(\Hom^{\bullet} (M,N)) \ar[r]  & \Hom^{\bullet} (M,N)_0 \ar[r] & B_{-1}(\Hom^{\bullet} (M,N)) \ar[r] & 0}$$

The map $Z_0(\Hom^{\bullet} (M,P)) \to Z_0(\Hom^{\bullet} (M,N))$ is epic since $ N \in {\underline {\mathfrak{Pr}}}_{\C(R)}^{-1}(M)$, so again $\Hom^{\bullet} (M,P)_0 \to \Hom^{\bullet} (M,N)_0$ is epic so we see that every morphism $M_n \to N_n$ factors through $P_n$ for every $n \in \mathbb{Z}$.
\end{proof} 

Though the fact that a complex $N$ belongs to the subprojectivity domain of another complex $M$ does not imply that the components of $N$ necessarily belong to the subprojectivity domains of the components of $M$, the answer is completely different if we ask about cycles of $N$ instead of components of $N$. We can see this in the following result.

\begin{lem}\label{lem-sph} 
Let $n\in \mathbb{Z}$, $N$ be a complex and $M$ be a module. If $N\in{\underline {\mathfrak{Pr}}}_{\C(R)}^{-1}(\underline{M}[n])$, then $Z_n(N) \in  {\underline {\mathfrak{Pr}}}_{R\Mod}^{-1}(M)$.
\end{lem}
\begin{proof} Let $f:M\to Z_n(N)$ be any morphism of modules and $\underline{f}:\underline{M}[n] \to N$ be the induced morphism of complexes. By assumption $\underline{f}$ factors as $$\xymatrix{\underline{M}[n] \ar[r]_{\alpha} \ar@/^1pc/[rr]^{\underline{f}} & P \ar[r]_{\beta} & N}$$ for some projective complex $P$. Then, $d_n^P \alpha_n=0$, so there exists a morphism $h:M \to Z_n(P)$ such that $\mu_n^P h =\alpha_n$. 

On the other side, the morphism $\beta$ induces a morphism $g:Z_n(P) \to Z_n(N)$ such that $\mu_n^N g=\beta_n \mu_n^P$. Then, we have $$\mu_n^N gh= \beta_n \mu_n^Ph=\beta_n \alpha_n= \underline{f}_n=\mu_n^N f,$$ that is, $f=gh$, so $f$ factors through the projective module $Z_n(P)$.
\end{proof}

Another natural question at this point is whether the inverse implication of Proposition \ref{prop-compon} is true or not. Namely, given two complexes $M$ and $N$, is the condition ``$N_{n} \in {\underline {\mathfrak{Pr}}}_{R\Mod}^{-1}(M_n)$ for every $n \in \mathbb{Z}$", sufficient to assure that $ N \in {\underline {\mathfrak{Pr}}}_{\C(R)}^{-1}(M)$? Again, this is not true in general since, for instance, for exact complexes it only holds over left hereditary rings (see Proposition \ref{prop-hered}).

\bigskip
We have studied so far the relation between subprojectivity and null-homotopic morphisms involving kernels of epimorphisms. We will now see that this relation can also be described without considering such kernels (Theorem \ref{thm-4-1}).\\

We start by characterizing contractible complexes in terms of subprojectivity. We need the following lemma.

\begin{lem}\label{lem-1}
Let $M$ be a complex, $N$ be a module and $n \in \mathbb{Z}$. Then, $\overline{N}[n] \in {\underline {\mathfrak{Pr}}}_{\C(R)}^{-1}(M)$ if and only if $N \in {\underline {\mathfrak{Pr}}}_{R\Mod}^{-1}(M_n)$.
\end{lem}
\begin{proof}
Suppose that $\overline{N}[n] \in {\underline {\mathfrak{Pr}}}_{\C(R)}^{-1}(M)$ and let $f:M_n \to N$ be a morphism of modules.  The induced morphism $\overline{f}:M \to \overline{N}[n]$ (that is, $\overline{f}_n=f$) factors through a projective complex $P$ by the hypothesis, so $f$ factors through the projective module $P_n$.

Conversely, let $f:M \to \overline{N}[n]$ be a morphism of complexes. Since $N \in {\underline {\mathfrak{Pr}}}_{R\Mod}^{-1}(M_n)$, the module morphism $f_n$ factors as $$\xymatrix{M_n \ar[r]_{\alpha} \ar@/^1pc/[rr]^{f_n} & P \ar[r]_{\beta} & N}$$ for some projective module $P$. Then, if we let $g:M \to \overline{P}[n]$ be the morphism of complexes with $g_n=\alpha$ and $g_{n+1}=\alpha d^M_{n+1}$, and $h: \overline{P}[n] \to \overline{N}[n]$ be the morphism of complexes with $h_n=h_{n+1}=\beta$, we clearly get that $f=hg$, hence $\overline{N}[n] \in {\underline {\mathfrak{Pr}}}_{\C(R)}^{-1}(M)$.
\end{proof}

\begin{prop}\label{prop-1}
Let $M$ be a complex and $(N_n)_{n \in \mathbb{Z}}$ be a family of modules. Then, $\oplus_{n \in \mathbb{Z}} \overline{N_n}[n] \in {\underline {\mathfrak{Pr}}}_{\C(R)}^{-1}(M)$ if and only if $N_n \in {\underline {\mathfrak{Pr}}}_{R\Mod}^{-1}(M_n)$ for every $n \in \mathbb{Z}$.
\end{prop}
\begin{proof} If $\oplus_{n \in \mathbb{Z}} \overline{N_n}[n] \in {\underline {\mathfrak{Pr}}}_{\C(R)}^{-1}(M)$ then $ \overline{N_n}[n] \in {\underline {\mathfrak{Pr}}}_{\C(R)}^{-1}(M)$ for every $n \in \mathbb{Z}$ since ${\underline {\mathfrak{Pr}}}_{\C(R)}^{-1}(M)$ is closed under direct summands (see\cite[Proposition 3.1]{SubprojAb}). Then, by Lemma \ref{lem-1} we get that for every $n\in\mathbb{Z}$, $N_n \in {\underline {\mathfrak{Pr}}}_{R\Mod}^{-1}(M_n)$.

Conversely, if $N_n \in {\underline {\mathfrak{Pr}}}_{R\Mod}^{-1}(M_n)$ for every $n \in \mathbb{Z}$ then $\overline{N_n}[n] \in {\underline {\mathfrak{Pr}}}_{\C(R)}^{-1}(M)$ for every $n \in \mathbb{Z}$ again by  Lemma \ref{lem-1}.

Now, let $f:M \to \oplus_{n \in \mathbb{Z}} \overline{N_n}[n]$ be a morphism of complexes and, for every $m$, choose an epimorphism $g^m: \overline{P_m}[m] \to \overline{N_m}[m]$ with $P_m$ a projective module.

If we let $$\pi^m : \oplus_{n \in \mathbb{Z}} \overline{N_n}[n] \to \overline{N_m}[m]$$ be the projection morphism, for any $m$  there exists a morphism $h^m:M \to \overline{P_m}[m]$ such that $\pi^m f=g^m h^m$.

But $\oplus_{n \in \mathbb{Z}} \overline{P_n}[n]$ coincides with $\prod_{n \in \mathbb{Z}} \overline{P_n}[n]$, so if we call $$\pi'^m : \oplus_{n \in \mathbb{Z}} \overline{P_n}[n] \to \overline{P_m}[m]$$ the projection morphism, we get a morphism $h :M \to \oplus_{n \in \mathbb{Z}} \overline{P_n}[n]$ such that $\pi'^m h=h^m$ for every $m$.

Therefore, for every $m \in\mathbb{Z}$ we have $$\pi^m f= g^m h^m =g^m\pi'^m h=\pi^m (\oplus g^n) h$$ so we see that $f= (\oplus g^n) h$. This means that $f$ factors through the projective complex $\oplus_{n \in \mathbb{Z}} \overline{P_n}[n]$ and so that $\oplus_{n \in \mathbb{Z}} \overline{N_n}[n] \in {\underline {\mathfrak{Pr}}}_{\C(R)}^{-1}(M)$.
\end{proof}
	
The following result characterizes subprojectivity in terms of factorization of morphisms through contractible complexes and through complexes in subprojectivity domains.  

\begin{prop}\label{pro-cont2}
Let $M$ and $N$ be two complexes. The following conditions are equivalent.
\begin{enumerate}
\item $N\in{\underline {\mathfrak{Pr}}}_{\C(R)}^{-1}(M)$.
\item Every morphism $M \to N$ factors through a complex of   ${\underline {\mathfrak{Pr}}}_{\C(R)}^{-1}(M)$. 
\item Every morphism $M \to N$ factors through a contractible complex $\oplus_{n \in \mathbb{Z}} \overline{X_n}[n]$ such that $X_n \in {\underline {\mathfrak{Pr}}}_{R\Mod}^{-1}(M_n)$ for every $n \in \mathbb{Z}$.
\end{enumerate}
\end{prop}
\begin{proof} $1.\Rightarrow 2.$ This is clear since every projective complex holds in ${\underline {\mathfrak{Pr}}}_{\C(R)}^{-1}(M)$.

$2.\Rightarrow 1.$ Let $f: M \to N$ be a morphism of complexes. By the hypothesis there exist two morphisms of complexes $\alpha: M \to L$ and $\beta:L\to N$ such that $f=\beta \alpha$ and that $L\in {\underline {\mathfrak{Pr}}}_{\C(R)}^{-1}(M)$. But then, $\alpha: M \to L$ factors through a projective complex $P$ so $f$ factors through $P$.

$1.\Rightarrow 3.$ Clear since every projective module holds in the subprojectivity domain of any module.

$3. \Rightarrow 2.$ Apply Proposition \ref{prop-1}.
\end{proof}

Notice that conditions $1.$ and $2.$ of Proposition \ref{pro-cont2} are equivalent in any abelian category with enough projectives.

\begin{lem}\label{lem-nul1}
Let $f: X \to Y$ be a null-homotopic morphism of complexes by a morphism $s$. If every morphism $s_n:X_n \to Y_{n+1}$ factors through a module $L_{n+1}$, then $ f:X\to Y$ factors through the contractible complex $\oplus_{n \in \mathbb{Z}} \overline{L_{n+1}}[n]$. In particular, $ f: X \to Y$ factors through the contractible complex $\oplus_{n \in \mathbb{Z}} \overline{Y_{n+1}}[n]$.
\end{lem}
\begin{proof}
Suppose that for any $n$ there exist two morphisms $\alpha_n:X_n\to L_{n+1}$ and $\beta_n:L_{n+1}\to Y_{n+1}$ such that $s_n= \beta_n \alpha_n$. Then, we have the situation $$\xymatrix{X_{n+1}\ar[rrr]^{d_{n+1}^X}\ar[rdd]_{f_{n+1}}&&&X_{n}\ar[rrr]^{d_{n}^X}\ar[rdd]_{f_n}\ar[lldd]^{s_{n}} \ar[lld]_{\alpha_{n}}&&& X_{n-1} \ar[rdd]_{f_{n-1}}\ar[lldd]^{s_{n-1}}\ar[lld]_{\alpha_{n-1}} \\ & L_{n+1}\ar[d]^{\beta_{n}}&&&L_n\ar[d]^{\beta_{n-1}}&&& \\ & Y_{n+1}\ar[rrr]^{d_{n+1}^Y}&&&Y_{n}\ar[rrr]^{d_{n}^Y}&&& Y_{n-1}}$$

Call $Z$ the complex $\oplus_{n \in \mathbb{Z}} \overline{L_{n+1}}[n]$ and consider, for every $n \in \mathbb{Z}$, the two morphisms of modules $h_n:L_{n+1}\oplus L_n\to Y_n$ given by $h_n(t,z)=d^Y_{n+1}\beta_n(t)+\beta_{n-1}(z)$, and $g_n:X_n\to L_{n+1}\oplus L_n$ given by $g_n=(\alpha_n,\alpha_{n-1}d_n^X)$. We claim that both $h:Z\to Y$ and $g:X\to Z$ are morphisms of complexes. 

For any $n \in \mathbb{Z}$, any $t \in L_{n+1}$ and any $z \in L_{n}$ we have $$h_{n-1}d^Z_n(t,z)=h_{n-1}(z,0)=d^Y_{n} \beta_{n-1}(z)=d^Y_{n}(d^Y_{n+1} \beta_n(t) + \beta_{n-1}(z))=d^Y_{n} h_n(t,z),$$ so $h$ is a morphism of complexes, and for any $n \in \mathbb{Z}$ we have $$g_{n-1}d^X_n=(\alpha_{n-1},\alpha_{n-2}d_{n-1}^X)d^X_n=(\alpha_{n-1}d^X_n, 0)=d^Z_{n}(\alpha_n,\alpha_{n-1}d_n^X)=d^Z_{n} g_n,$$ so $g$ is also a morphism of complexes.

Now we see that $f=hg$ since for any $n \in \mathbb{Z}$ and any $x \in X_n$ we have $$h_n g_n(x)=d^Y_{n+1} \beta_n \alpha_n(x) + \beta_{n-1}\alpha_{n-1}d_n^X(x)= d^{Y}_{n+1}s_n(x)+s_{n-1}d_n^{X}(x)=f_n(x).$$

Therefore, $f:X\to Y$ factors through the contractible complex $Z=\oplus_{n \in \mathbb{Z}} \overline{L_{n+1}}[n]$.
\end{proof}

\begin{thm} \label{thm-4-1}
Let $M$ and $N$ be two complexes such that $N_{n+1} \in {\underline {\mathfrak{Pr}}}_{R\Mod}^{-1}(M_n)$ for every $n \in \mathbb{Z}$. Then, $ N \in {\underline {\mathfrak{Pr}}}_{\C(R)}^{-1}(M)$ if and only if $\Hom_{\K(R)}(M,N)=0.$
\end{thm}
\begin{proof}
If $N \in {\underline {\mathfrak{Pr}}}_{\C(R)}^{-1}(M)$ then every morphism $M \to N$ factors through a projective complex. Then, by \cite[Corollary 3.5]{Gillespie2}, $\Hom_{\K(R)}(M,N)=0$. 

Conversely, if $\Hom_{\K(R)}(M,N)=0$ then, by Lemma \ref{lem-nul1}, every morphism $ M \to N$ factors through the contractible complex $\oplus_{n \in \mathbb{Z}} \overline{N_{n+1}}[n]$. Therefore, $ N \in {\underline {\mathfrak{Pr}}}_{\C(R)}^{-1}(M)$ by Proposition \ref{pro-cont2}.
\end{proof}

The following example shows that the condition $N_{n+1} \in {\underline {\mathfrak{Pr}}}_{R\Mod}^{-1}(M_n)$ for every $n \in \mathbb{Z}$ in Theorem \ref{thm-4-1} cannot be removed in general.

\begin{ex}\label{ex-2main1}
Let $X$ be any non-projective module and choose any other module $Y$ out of the the subprojectivity domain of $X$ (such modules exist over any non semisimple ring). It is clear that $\Hom_{\K(R)}(\overline{X}, \overline{Y})=0$ and, by Lemma \ref{lem-1}, that $\overline{Y}\notin {\underline {\mathfrak{Pr}}}_{\C(R)}^{-1}(\overline{X})$.
\end{ex}

Given two complexes $M$ and $N$, it is clear that the condition ``$N_{n+1} \in {\underline {\mathfrak{Pr}}}_{R\Mod}^{-1}(M_n)$ for every $n \in \mathbb{Z}$" is not enough in general to get $N \in {\underline {\mathfrak{Pr}}}_{\C(R)}^{-1}(M)$. For instance, if $R$ is semisimple and $M$ is not exact (so $M$ is not a projective complex), then for sure we can find complexes not in ${\underline {\mathfrak{Pr}}}_{\C(R)}^{-1}(M)$.

In the following result we prove that this condition suffices for exact complexes if and only if the ring $R$ is semisimple.

\begin{prop}\label{prop-semisimple}
The following conditions are equivalent.
\begin{enumerate}
\item $R$ is semisimple.
\item  For every complex $M$ and every exact complex $N$, if $N_{n+1} \in {\underline {\mathfrak{Pr}}}_{R\Mod}^{-1}(M_n)$ for every $n \in \mathbb{Z}$, then $ N \in {\underline {\mathfrak{Pr}}}_{\C(R)}^{-1}(M)$.
\item For every module $M$ and every exact complex $N$, if there exists $n \in \mathbb{Z}$ such that $N_{n+1} \in {\underline {\mathfrak{Pr}}}_{R\Mod}^{-1}(M)$, then $ N \in {\underline {\mathfrak{Pr}}}_{\C(R)}^{-1}(\underline{M}[n])$.
\end{enumerate}
\end{prop}
\begin{proof} $1.\Rightarrow 2.$ Every exact complex $N$ is projective so $N\in {\underline {\mathfrak{Pr}}}_{\C(R)}^{-1}(M)$ for every complex $M$.

$2. \Rightarrow 3.$ Clear.

$3. \Rightarrow 1.$ Let $M$ be a module and ${\cal P}$ be a projective resolution of $M$. Then, ${\cal P}_1\in{\underline {\mathfrak{Pr}}}_{R\Mod}^{-1}(\underline{M})$ and so ${\cal P}\in{\underline {\mathfrak{Pr}}}_{\C(R)}^{-1}(\underline{M})$ by assumption. Then, by Lemma \ref{lem-sph}, $M=Z_{0}({\cal P}) \in  {\underline {\mathfrak{Pr}}}_{R\Mod}^{-1}(M)$. This means that  $M$ is projective and therefore that $R$ is semisimple.
\end{proof}

Given two complexes $M$ and $N$, it is natural to ask whether $ N \in {\underline {\mathfrak{Pr}}}_{\C(R)}^{-1}(M)$ implies that $N_{n+1} \in {\underline {\mathfrak{Pr}}}_{R\Mod}^{-1}(M_n)$ for every $n \in \mathbb{Z}$. This is not true in general: take any non-projective module $X$ and choose any other module $Y$ out of the the subprojectivity domain of $X$ (such modules exist over any non semisimple ring). Then, the complex $\underline{Y}[2]$ belongs to ${\underline {\mathfrak{Pr}}}_{\C(R)}^{-1}(X)$ since $\Hom_{\C(R)}(\overline{X}, \underline{Y}[2])=0$, but $\underline{Y}[2]_2=Y\notin{\underline {\mathfrak{Pr}}}_{R\Mod}^{-1}(\overline{X})$.

However, if we add the condition ``$N\in {\underline {\mathfrak{Pr}}}_{\C(R)}^{-1}(M[1])$", then Proposition \ref{prop-compon} says that $N_{n+1} \in {\underline {\mathfrak{Pr}}}_{R\Mod}^{-1}(M_n)$ for every $n \in \mathbb{Z}$.\\

Inspired by Proposition \ref{prop-compon}, we give the following result.
	
\begin{prop} \label{prop-shift}
Let $M$ and $N$ be two complexes. The following statements are equivalent.
\begin{enumerate}
\item $N\in {\underline {\mathfrak{Pr}}}_{\C(R)}^{-1}(M[n])$ for every $n \in \mathbb{Z}$.
\item For every $i,j \in \mathbb{Z}$,  $N_i \in {\underline {\mathfrak{Pr}}}_{R\Mod}^{-1}(M_j)$, and  $\Hom_{\K(R)}(M[n],N)=0$ for every $n\in \mathbb{Z}$.
\end{enumerate}
\end{prop}
\begin{proof} Apply Proposition \ref{prop-compon} and Theorem \ref{thm-4-1}.
\end{proof}
	
Now, we give some applications of Proposition \ref{prop-shift}. Namely, given any module $M$, Proposition \ref{prop-shift} can be used to study the subprojectivity domain of the complexes $\oplus_{n\in \mathbb{Z}}\overline{M}[n]$ (Proposition \ref{lem-discs}) and $\oplus_{n\in \mathbb{Z}}\underline{M}[n]$ (Proposition \ref{prop-sph}).

\begin{prop}\label{lem-discs}
Let $N$ be a complex and $M$ be a module. The following statements are equivalent.
\begin{enumerate}
\item $N \in {\underline {\mathfrak{Pr}}}_{\C(R)}^{-1}(\oplus_{n\in \mathbb{Z}}\overline{M}[n])$.
\item $N \in {\underline {\mathfrak{Pr}}}_{\C(R)}^{-1}(\overline{M}[n])$ for every $n\in \mathbb{Z}$.
\item $N_{n} \in  {\underline {\mathfrak{Pr}}}_{R\Mod}^{-1}(M)$ for every $n\in \mathbb{Z}$. 
\end{enumerate}
\end{prop}
\begin{proof} $1. \Leftrightarrow 2.$ Clear by \cite[Proposition 2.16]{SubprojAb}.

$2. \Leftrightarrow 3.$ Clear by Proposition \ref{prop-shift} since $\Hom_{\K(R)}(\overline{M}[n],N)=0$ for every $n \in \mathbb{Z}$.
\end{proof}

\begin{prop}\label{prop-sph}
Let $N$ be a complex and $M$ be a module. The following statements are equivalent.
\begin{enumerate}
\item $N\in {\underline {\mathfrak{Pr}}}_{\C(R)}^{-1}(\oplus_{n\in \mathbb{Z}}\underline{M}[n])$.
\item $N\in {\underline {\mathfrak{Pr}}}_{\C(R)}^{-1}(\underline{M}[n])$ for every $n\in \mathbb{Z}$.
\item $N$ is $\Hom_R(M,-)$-exact and $ N_n \in {\underline {\mathfrak{Pr}}}_{R\Mod}^{-1}(M)$ for every $n\in \mathbb{Z}$.
\end{enumerate}
\end{prop}
\begin{proof}
$1.\Leftrightarrow 2.$ Clear by \cite[Proposition 2.16]{SubprojAb}.

$2.\Rightarrow 3.$ By Proposition \ref{prop-shift} we know that $N_n \in {\underline {\mathfrak{Pr}}}_{R\Mod}^{-1}(M)$ for every $n\in \mathbb{Z}$ and that $H_n(\Hom^\bullet(\underline{M},N))=\Hom_{\K(R)}(\underline{M}[n],N)=0$. But $\Hom^\bullet(\underline{M},N) \cong \Hom_R(M,N)$ so $\Hom_R(M,N)$ is exact.

$3.\Rightarrow 2.$ Let $n \in \mathbb{Z}$ and $f:\underline{M}[n] \to N$ be a morphism of complexes. Since $d^N_nf_n=0$ we get that $f_n \in \Ker(\Hom_R(M,d^N_{n}))=\Im(\Hom_R(M,d^N_{n+1}))$, so there exists a morphism of modules $g:M \to N_{n+1}$ such that $d^N_{n+1}g=f_n$. Thus, $f$ is null-homotopic and $\Hom_{\K(R)}(\underline{M}[n],N)=0$ for every $n \in \mathbb{Z}$. Proposition \ref{prop-shift} says then that $ N \in {\underline {\mathfrak{Pr}}}_{\C(R)}^{-1}(\underline{M}[n])$ for every $n \in \mathbb{Z}$.
\end{proof}

If let $M=R$ in Proposition \ref{prop-sph}, the condition ``$N$ is $\Hom_R(M,-)$-exact" means that $N$ is exact. This leads to the following characterization of exact complexes in terms of subprojectivity. 

\begin{cor}\label{cor-exac}
Let $N$ be a complex. The following assertions are equivalent.
\begin{enumerate}
\item $N$ is exact.
\item $N \in {\underline {\mathfrak{Pr}}}_{\C(R)}^{-1}(\oplus_{n\in \mathbb{Z}}\underline{R}[n])$.
\item $N \in {\underline {\mathfrak{Pr}}}_{\C(R)}^{-1}(\underline{R}[n])$ for every $n\in \mathbb{Z}$.
\end{enumerate}
\end{cor}

There is now a natural question  which comes to mind after Corollary \ref{cor-exac}: we have described how the subprojectivity domain of the set of complexes $\{\underline{R}[n], n\in \mathbb{Z}\}$ is, so, what about the subprojectivity domain of each of the complexes $\underline{R}[n]$? Can we describe them as well? 

Given a complex $N$, we know, by Theorem \ref{thm-4-1}, that $N \in {\underline {\mathfrak{Pr}}}_{\C(R)}^{-1}(\underline{R}[n])$ if and only if $\Hom_{\K(R)}(\underline{R}[n],N)=0$. But, $$\Hom_{\K(R)}(\underline{R}[n],N)=H_n(\Hom^\bullet(\underline{R},N))=H_n(\Hom_R(R,N))=H_n(N).$$ So, the condition $\Hom_{\K(R)}(\underline{R}[n],N)=0$ is equivalent to $H_n(N)=0$. We state this fact in the following proposition.

\begin{prop}\label{prop-spherR}
Let $N$ be a complex and $n\in \mathbb{Z}$. The following assertions are equivalent.
\begin{enumerate}
\item $N \in {\underline {\mathfrak{Pr}}}_{\C(R)}^{-1}(\underline{R}[n])$.
\item $\Hom_{\K(R)}(\underline{R}[n],N)=0$.
\item $H_n(N)=0$.
\end{enumerate}	
\end{prop}

Now, with Proposition \ref{prop-spherR} in hand, it is easy to see that subprojectivity domains are not closed under kernels of epimorphisms in general.  

\begin{ex}\label{exmp-1-spherR} 
  Consider the short exact sequence of complexes $$0 \to \underline{R} \to \overline{R} \to \underline{R}[1] \to 0.$$ It is clear by Proposition \ref{prop-spherR} that $\underline{R}[1]$ and $\overline{R}$ both hold in ${\underline {\mathfrak{Pr}}}_{\C(R)}^{-1}(\underline{R})$, but $\underline{R}$ does not. Therefore, the subprojectivity domain of $\underline{R}$ is not closed under kernels of epimorphisms.
\end{ex}

Moreover, Proposition \ref{prop-spherR} helps us to answer a question raised in Remark \ref{rem-cond1}.  Precisely, it is understood by the equivalence $(1\Leftrightarrow 4)$ in    Theorem \ref{thm-4-2} that   the second assertion remains equivalent to the first assertion even if we replace the condition   ``$P$ is projective'' with $P\in {\underline {\mathfrak{Pr}}}_{\C(R)}^{-1}(M)$. However, this fact does not hold true. Namely, the following example shows that if we replace  ``$P$ is projective''  with $P\in {\underline {\mathfrak{Pr}}}_{\C(R)}^{-1}(M)$ in assertion 2, the equivalent does not hold.

\begin{ex}\label{exmp-2-spherR}  
Let $0\to N_3 \to N_2 \to N_1 \to 0$ be a short exact sequence of modules such that $N_3\neq 0$ and let $X_i:= \overline{N_i}\oplus\underline{N_i}[-1]$ for $i\in \{1,2,3\}$. Then, we have an induced exact sequence of complexes $0\to X_3 \to X_2 \to X_1 \to 0$.

Moreover, we see that for $i \in \{1,2,3\}$ it holds that $H_0(X_i)=H_0(\overline{N_i})\oplus H_0(\underline{N_i}[-1])= 0$ and that $H_{-1}(X_i)=H_{-1}(\overline{N_i})\oplus H_{-1}(\underline{N_i}[-1])=H_{-1}(\underline{N_i}[-1])= N_i$. Therefore, we can assert that $N_1, N_2 \in {\underline {\mathfrak{Pr}}}_{\C(R)}^{-1}(\underline{R})$ and that $\Hom_{\K(R)}(\underline{R}[-1],X_3)\neq 0$ (see Proposition \ref{prop-spherR}).
\end{ex}

\section{Applications} \label{sect:aplic}
\hskip .5cm Recall that a complex $P$ is said to be DG-projective if its components are projective and $\Hom^\bullet(P,E)$ is exact for every exact complex $E$. In \cite[Proposition 2.3]{Yang2} it is proved that, under certain conditions, a ring is left hereditary if and only if every subcomplex of a DG-projective complex is DG-projective. Among these conditions, the authors included: ``Every exact complex of projective modules is projective''. In this section, using the properties of subprojectivity domains, we will show that the latter equivalence holds without the mentioned assumption.

\begin{prop}\label{prop-hered}
For any ring $R$, the following statements are equivalent.
\begin{enumerate}
\item $R$ is left hereditary.
\item For every complex $M$ and every exact complex $N$, if $N_{n} \in {\underline {\mathfrak{Pr}}}_{R\Mod}^{-1}(M_n)$ for every $n\in \mathbb{Z}$, then $ N \in {\underline {\mathfrak{Pr}}}_{\C(R)}^{-1}(M)$.
\item For every module $M$ and every exact complex $N$, if there exists $n \in \mathbb{Z}$ such that $N_{n} \in {\underline {\mathfrak{Pr}}}_{R\Mod}^{-1}(M)$, then $ N \in {\underline {\mathfrak{Pr}}}_{\C(R)}^{-1}(\underline{M}[n])$.
\item Every subcomplex of a DG-projective complex is DG-projective.
\end{enumerate}
\end{prop}
\begin{proof} $1.\Rightarrow 2.$ If $0\to K \to P \to N\to 0$ is a short exact sequence of complexes with $P$ projective then $K$ is exact ($P$ and $N$ are exact) and all cycles $Z_n(K)$ are projective by 1. Therefore, $K$ is projective and then $\Hom_{\K(R)} (M[-1],K)=0$ by \cite[Corollary 3.5]{Gillespie2}, so $N\in{\underline {\mathfrak{Pr}}}_{\C(R)}^{-1}(M)$ by Theorem \ref{thm-4-2}.

$2. \Rightarrow 3.$ Clear.

$3. \Rightarrow 1.$ Let $Q$ be a projective module and $Y$ be any submodule of $Q$. Let us prove that  ${\underline {\mathfrak{Pr}}}_{R\Mod}^{-1}(Y)=R\Mod$. For let $X$ be a module and consider the exact complex $${\cal C}:\cdots \to 0 \to X \to E(X) \to C \to 0 \to \cdots$$ ($E(X)$ in the $0$ position). By \cite[Lemma 2.2]{Dur}, $E(X) \in {\underline {\mathfrak{Pr}}}_{R\Mod}^{-1}(Y)$, so ${\cal C}\in{\underline {\mathfrak{Pr}}}_{\C(R)}^{-1}(\underline{Y})$ by assumption. Then, $X \in {\underline {\mathfrak{Pr}}}_{R\Mod}^{-1}(Y)$ by Lemma \ref{lem-sph}.

$1.\Rightarrow 4.$ Let $P$ be a DG-projective complex and $Q$ a subcomplex of $P$. Then, every module $Q_n$ is projective by condition $1.$

Now, let $E$ be an exact complex and let us prove that $\Hom^\bullet(Q,E)$ is exact. 

Let $0\to E \to I \to C\to 0$ be a short exact sequence of complexes with $I$ injective. Since every module $Q_n$ is projective we get that for every $n\in \mathbb{Z}$, $ \Hom^{\bullet} (Q,I)_n \to \Hom^{\bullet} (Q,C)_n$ is epic, and for every $i,j \in \mathbb{Z}$,  $C_i \in {\underline {\mathfrak{Pr}}}_{R\Mod}^{-1}(Q_j)$. Then, by condition 2. we get that $C \in {\underline {\mathfrak{Pr}}}_{\C(R)}^{-1}(Q[n])$ for every $n\in \mathbb{Z}$ ($C$ is exact since $I$ and $E$ are exact), so for every $n\in \mathbb{Z}$, $Z_{n}(\Hom^{\bullet} (Q,I))  \to Z_{n}(\Hom^{\bullet} (Q,C))$ is epic. Therefore, for every $n\in \mathbb{Z}$ the two first columns of the commutative diagram with exact rows $$\xymatrix{&0\ar[d]& 0\ar[d]& 0\ar[d]& \\ 0 \ar[r] & Z_{n}(\Hom^{\bullet} (Q,E)) \ar[r] \ar[d]  & \Hom^{\bullet} (Q,E)_n \ar[r] \ar[d] & B_{n-1}(\Hom^{\bullet} (Q,E)) \ar[r] \ar[d]& 0 \\ 0 \ar[r] & Z_{n}(\Hom^{\bullet} (Q,I)) \ar[r] \ar[d]  & \Hom^{\bullet} (Q,I)_n \ar[r] \ar[d] & B_{n-1}(\Hom^{\bullet} (Q,I)) \ar[r] \ar[d]& 0 \\ 0 \ar[r] & Z_{n}(\Hom^{\bullet} (Q,C)) \ar[r] \ar[d]  & \Hom^{\bullet} (Q,C)_n \ar[r] \ar[d] & B_{n-1}(\Hom^{\bullet} (Q,C)) \ar[r] \ar[d]& 0 \\ &0&0&0}$$ are exact, so the third is also exact. 

Now consider, for every $n\in \mathbb{Z}$, the commutative diagram with exact rows $$\xymatrix{ &0\ar[d]& 0\ar[d]& 0\ar[d]& \\ 0 \ar[r] & B_{n}(\Hom^{\bullet} (Q,E)) \ar[r] \ar[d]  &  Z_{n}(\Hom^{\bullet} (Q,E)) \ar[r] \ar[d] & H_{n}(\Hom^{\bullet} (Q,E)) \ar[r] \ar[d]& 0 \\ 0 \ar[r] & B_{n}(\Hom^{\bullet} (Q,I)) \ar[r] \ar[d]  &  Z_{n}(\Hom^{\bullet} (Q,I)) \ar[r] \ar[d] & H_{n}(\Hom^{\bullet} (Q,I)) \ar[r] \ar[d]& 0 \\ 0 \ar[r] & B_{n}(\Hom^{\bullet} (Q,C)) \ar[r] \ar[d]  &  Z_{n}(\Hom^{\bullet} (Q,C)) \ar[r] \ar[d] & H_{n}(\Hom^{\bullet} (Q,C)) \ar[r] \ar[d]& 0 \\ &0&0&0}$$ 

The first and second columns are exact, so the third one is also exact. But, for every $n\in \mathbb{Z}$, $H_{n}(\Hom^{\bullet} (Q,I))=\Hom_{\K(R)} (Q[n],I)=0$ by \cite[Corollary 3.5]{Gillespie2} since $I$ is contractible.

$4. \Rightarrow 1.$ Let $Q$ be a projective module and $Y$ a submodule of $Q$. Since $\underline{Y}$ is a subcomplex of the DG-projective complex $\underline{Q}$, $\underline{Y}$ must be DG-projective by assumption, so $Y$ is projective.
\end{proof}

It is a well-known fact that a ring is left semi-hereditary if and only if it is left coherent and every submodule of a flat module is flat (i.e.,  the weak global dimension of the ring is at most $1$). Using subprojectivity we can prove a similar result in the categories of complexes. Namely, a ring is left semi-hereditary if and only if it is left coherent and every subcomplex of a DG-flat complex is DG-flat (Corollary \ref{cor-semi-her}). This is so because rings for which subcomplexes of DG-flat complexes are DG-flat are precisely those of weak global dimension at most 1 (Proposition \ref{prop-Weak-dim}).

We first recall that a complex is finitely presented if it is bounded and has finitely presented components (see \cite[Lemma 4.1.1]{JR}). Recall also  that the subprojectivity domain of the class of all finitely presented complexes (respectively, modules) is the class of all flat complexes (respectively, modules) (see \cite[Proposition 2.18]{SubprojAb}). Finally, recall that a complex $F$ is said to be DG-flat if $F_n$ is flat for every $n \in \mathbb{Z}$ and the complex $E \otimes^\bullet F$ is exact for any exact complex $E$ of right $R$-modules (see \cite{Avr}).

\begin{prop}\label{prop-Weak-dim}
For any ring $R$, the following assertions are equivalent.
\begin{enumerate}
\item The weak global dimension of $R$ is  at most $1$.
\item For every finitely presented complex $M$ and every exact complex $N$, if $N_{n} \in {\underline {\mathfrak{Pr}}}_{R\Mod}^{-1}(M_n)$ for every $n \in \mathbb{Z}$, then $N \in {\underline {\mathfrak{Pr}}}_{\C(R)}^{-1}(M)$.
\item For every finitely presented module $M$ and every exact complex $N$, if there exists $n \in \mathbb{Z}$ such that $N_{n} \in {\underline {\mathfrak{Pr}}}_{R\Mod}^{-1}(M)$, then $ N \in {\underline {\mathfrak{Pr}}}_{\C(R)}^{-1}(\underline{M}[n])$.
\item Every subcomplex of a DG-flat complex is DG-flat.
\end{enumerate}
\end{prop}
\begin{proof} $1. \Rightarrow 2.$ Consider a short exact sequence of complexes $0\to K \to P \to N\to 0$ with $P$ projective. Since all cycles $Z_n(P)$ are projective, every cycle $Z_n(K)$ is flat by assumption. Then, $K$ is flat ($K$ is exact since $P$ and $N$ are), so $K\in {\underline {\mathfrak{Pr}}}_{\C(R)}^{-1}(M[-1])$ and hence $\Hom_{\K(R)} (M[-1],K)=0$ by \cite[Corollary 3.5]{Gillespie2}. Therefore, $ N \in {\underline {\mathfrak{Pr}}}_{\C(R)}^{-1}(M)$ by Theorem \ref{thm-4-2}.

$2. \Rightarrow 3.$ Clear.

$3. \Rightarrow 1.$ Let $X$ be a submodule of a flat module $F$. Let us prove that $X \in {\underline {\mathfrak{Pr}}}_{R\Mod}^{-1}(M)$ for every finitely presented module $M$. For let $M$ be a finitely presented module and consider the exact complex $${\cal F}:\cdots \to 0 \to X \to F \to C \to 0 \to \cdots$$ with $F$ in the $0$-position.

Since $F \in {\underline {\mathfrak{Pr}}}_{R\Mod}^{-1}(M)$, ${\cal F}\in {\underline {\mathfrak{Pr}}}_{\C(R)}^{-1}(\underline{M})$ by assumption, and then $X \in {\underline {\mathfrak{Pr}}}_{R\Mod}^{-1}(M)$ by Lemma \ref{lem-sph}.

$1.\Rightarrow 4.$ Let $F$ be a DG-flat complex, $N$ be a subcomplex of $F$ and $ P \to N$ be an epic quasi-isomorphism with $P$ DG-projective. To prove that $N$ is DG-flat it is sufficient to prove that for every finitely presented complex $M$,  $\Hom_{\C(R)}(M,P) \to \Hom_{\C(R)}(M,N)$ is epic (see \cite[Proposition 6.2]{Cris}). For let $f:M \to N$ be a morphism of complexes with $M$ finitely presented and consider the following pullback diagram $$\xymatrix{ 0\ar[r]& E\ar[r]\ar@{=}[d]& D\ar[r] \ar[d]& M \ar[r]\ar[d]^f&0\\ 0\ar[r]& E\ar[r]& P\ar[r]& N \ar[r]&0}$$

Every module $N_n$ is flat by $1$, so $N_{n} \in {\underline {\mathfrak{Pr}}}_{R\Mod}^{-1}(M_n)$ for every $n \in \mathbb{Z}$, and hence the short exact sequence $0 \to E \to D \to M \to 0$ splits at the module level by Proposition \ref{prop-pull} since for every $n\in \mathbb{Z}$, $D_n$ is a pullback (see Lemma \ref{lem-pull}). Then, the sequence  $ 0 \to E \to D \to M \to 0$ is equivalent to a short exact sequence $ 0\to E\to M(g)\to M \to 0$ where $M(g)$ is the mapping cone of a morphism $g: M[-1] \to E$ (see \cite[Section 3.3]{Enochs}). 

Now, every module $E_n$ is flat by condition $1$. So,  $E_{n} \in {\underline {\mathfrak{Pr}}}_{R\Mod}^{-1}(M_{n+1})$ for every $n \in \mathbb{Z}$. Thus, $E \in {\underline {\mathfrak{Pr}}}_{\C(R)}^{-1}(M[-1])$ by condition $2$ and then $\Hom_{\K(R)}(M[-1],E)=0$ by \cite[Corollary 3.5]{Gillespie2}.

In particular, $g: M[-1] \to E$ is null-homotopic so the sequence $ 0\to E\to M(g)\to M \to 0$ splits (see \cite[Proposition 3.3.2]{Enochs})and then the sequence $0\to E\to D\to M \to 0$ splits. Therefore,  $f$ clearly factors through $P \to N$.

$4. \Rightarrow 1.$ Let $F$ be a flat module and $Y$ a submodule of $F$. Then $\underline{Y}$ is a subcomplex of the DG-flat complex $\underline{F}$, so $\underline{Y}$ is also DG-flat by assumption and therefore $Y$ is flat.
\end{proof}

\begin{cor}\label{cor-semi-her}
For any ring $R$ the following statements are equivalent.
\begin{enumerate}
\item $R$ is left semi-hereditary.
\item $R$ is left coherent and for every finitely presented complex $M$ and every exact complex $N$, if $N_{n} \in {\underline {\mathfrak{Pr}}}_{R\Mod}^{-1}(M_n)$ for every $n \in \mathbb{Z}$, then $ N \in {\underline {\mathfrak{Pr}}}_{\C(R)}^{-1}(M)$.
\item $R$ is left coherent and for every finitely presented module $M$ and every exact complex $N$, if there exists $n \in \mathbb{Z}$ such that $N_{n} \in {\underline {\mathfrak{Pr}}}_{R\Mod}^{-1}(M)$, then $ N \in {\underline {\mathfrak{Pr}}}_{\C(R)}^{-1}(\underline{M}[n])$.
\item $R$ is left coherent and every subcomplex of a DG-flat complex is DG-flat.
\end{enumerate}
\end{cor}
\bigskip

\noindent\textbf{Acknowledgment:} The second and fourth authors were partially supported by Ministerio de Econom\'{\i}a y Competitividad, grant reference 2017MTM2017-86987-P  and  P20-00770 grant from Junta de Andaluc\'{\i}a.

Driss Bennis:    Faculty of Sciences, Mohammed V University in Rabat, Rabat, Morocco.

\noindent e-mail address: driss.bennis@fsr.um5.ac.ma; driss$\_$bennis@hotmail.com

J. R. Garc\'{\i}a Rozas: Departamento de  Matem\'{a}ticas,
Universidad de Almer\'{i}a, 04071 Almer\'{i}a, Spain.

\noindent e-mail address: jrgrozas@ual.es

Hanane Ouberka:    Faculty of Sciences, Mohammed V University in Rabat, Rabat, Morocco.

\noindent e-mail address: hanane$\_$ouberka@um5.ac.ma; ho514@inlumine.ual.es 

Luis Oyonarte: Departamento de  Matem\'{a}ticas, Universidad de
Almer\'{i}a, 04071 Almer\'{i}a, Spain.

\noindent e-mail address: oyonarte@ual.es

\end{document}